\documentclass[11pt]{amsart}

\usepackage[margin=1in]{geometry}
\usepackage{amsmath,amsthm,amssymb}
\usepackage{bbm}
\usepackage{mathrsfs}
\usepackage{enumerate}
\usepackage{graphicx}
\usepackage{graphics}
\usepackage{caption} 
\usepackage{subcaption} 
\usepackage{thmtools}
\usepackage{thm-restate}
\usepackage{hyperref}
\usepackage{float}
\usepackage{color}
\usepackage[normalem]{ulem}
\usepackage [english]{babel}
\usepackage [autostyle, english = american]{csquotes}
\MakeOuterQuote{"}

\newtheorem{thm}{Theorem}

\newtheorem{proposition}[thm]{Proposition}

\newtheorem{lemma}[thm]{Lemma}

\theoremstyle{definition}

\newtheorem{alg}[thm]{Algorithm}

\theoremstyle{remark}
\newtheorem{rem}[thm]{Remark}

\title{Polynomial algorithm for alternating link equivalence}
\author{Touseef Haider and Anastasiia Tsvietkova}
\date{}

\begin{document}

\maketitle

\begin{abstract}
Link equivalence up to isotopy in a 3-space is the problem that lies at the root of knot theory, and is important in 3-dimensional topology and geometry. We consider its restriction to alternating links, given by two alternating diagrams with $n_1$ and $n_2$ crossings, and show that this problem has polynomial algorithm in terms of $max\{n_1, n_2\}$. For the proof, we use Tait flyping conjectures, observations stemming from the work of Lackenby, Menasco, Sundberg and Thistlethwaite on alternating links, and algorithmic complexity of some problems from graph theory and topological graph theory. 
\end{abstract}

\section{Introduction}
\label{sec:introduction}

One link can be represented by many diagrams, \textit{i.e.} link projections onto a plane together with information about over and underpasses. It is non-trivial to determine if two link diagrams represent the same link up to isotopy in $\mathbb{R}^3$ or $S^3$. The following problem lies at the heart of classical knot theory, and was notably mentioned by Turing \cite{Turing} for one particular link, called the unknot.

\vspace{0.08in}
\textsc{link equivalence.} \textit{Given two link diagrams, do they represent the same link?}
\vspace{0.08in}

In recent years, interest in the algorithmic complexity of problems from 3-dimensional topology and knot theory has been growing. Some problems were proved to be computationally hard, i.e. NP-hard (\cite{AHT, KT, KoTs, MRST1, MRST2, MSS, Lackenby1}). At the same time, almost nothing is known about the existence of such computationally easy non-trivial problems, i.e. the existence of polynomial algorithms in knot theory and 3-manifold topology. Here, we address restriction of \textsc{link equivalence} to one infinite family, alternating links, and prove that this restriction is in class P (also called PTIME), when the input is two alternating link diagrams. Here is the restricted problem:

\vspace{0.08in}
\textsc{alternating link equivalence.} \textit{Given two alternating diagrams of a link, do they represent the same alternating link?}
\vspace{0.08in}

In 1960's and 1970's, \textsc{link equivalence} was proven to be decidable by Haken and Hemion \cite{Hemion, Haken} using theory of normal surfaces. More recent proof, from 2011, is by Coward and Lackenby \cite{CoLack}, and uses Reidemeister moves. Such diagrammatic moves were proved to connect any two diagrams of the same link in 1920's by Alexander, Briggs and Reidemeister \cite{AlBr, Reid}. Another proof that \textsc{link equivalence} is decidable can be obtained from the problem of 3-manifold homeomorphism, by passing from a link diagram to the triangulation of the link complement in 3-sphere. As detailed in the work of Kuperberg \cite{Kuperberg}, the 3-manifold homeomorphism problem is known to be decidable due to the proof of Geometrization (conjectured and partly proven by Thurston \cite{Thurston}, and completed by Perelman based on the program of Hamilton \cite{MoTsz}). While links are not uniquely determined by their complements, one can check whether there is a homeomorphism between the link exteriors taking meridians to meridians (see Corollary 6.1.4 in Matveev's book \cite{Matveev}), and this is enough for link isotopy. There is also an alternative solution to the homeomorphism problem for hyperbolic link complements by Dahmani and Groves \cite{DaGro}, based on work of Sela \cite{Sela} on isomorphism problem for hyperbolic groups. Meanwhile, practical link tabulation, often using supercomputers, is underway \cite{HTWTab, BurtonTab, ThisTab}. 

Alternating links are an important class of links: for example, they can be characterized purely topologically \cite{Howie, Greene}, and their topological and geometric properties can often be seen or computed directly from a minimal crossing link diagram \cite{Menasco1, Lackenby, TT}. Over a century ago, Tait proposed the following conjecture: given prime alternating reduced diagrams $D_1$, $D_2$ of a link, one can transform $D_1$ into $D_2$ (up to planar isotopy) by a sequence of diagrammatic moves called flypes. A flype is depicted in Figure \ref{flype}. The shaded disc labelled $T$ represents a 2-tangle in the terminology of Conway (\textit{i.e.} a tangle with four free ends). In performing a flype, $T$ is turned by 180 degrees so that the crossing to its left is removed by twisting this tangle, while a new crossing is created to its right. The conjecture was proved in full generality by Menasco and Thistlethwaite as Theorem 1 in \cite{MT}. We will further call it Tait flyping theorem.

\begin{figure}
\begin{subfigure}[b]{0.49 \textwidth}
\centering
\includegraphics[scale=0.6]{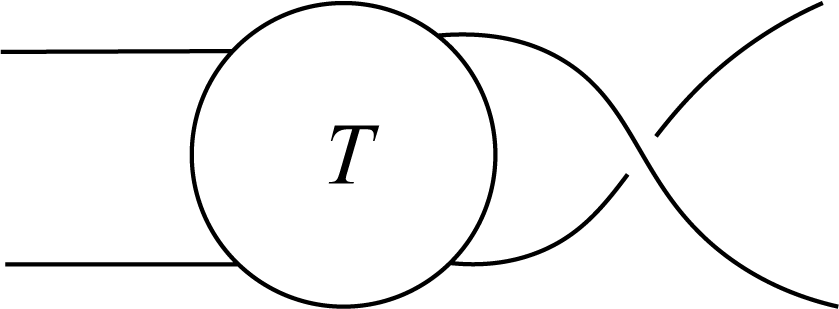}
\subcaption{}
\label{Flype1}
\end{subfigure}
\begin{subfigure}[b]{0.49 \textwidth}
\centering
\includegraphics[scale=0.6]{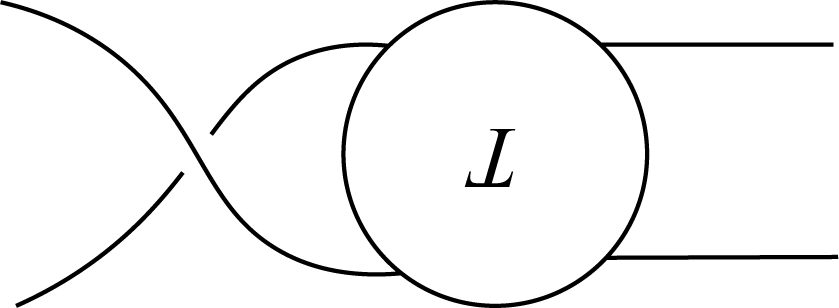}
\subcaption{}
\label{Flype2}
\end{subfigure}
\caption{A flype involving a 2-tangle $T$.}
\label{flype}
\end{figure}

The conjecture that \textsc{alternating link equivalence} is polynomial has circulated for a while. It has also been known that an existence of a polynomial algorithm for connected prime alternating reduced diagrams implies the existence of such algorithm for all alternating diagrams as a corollary of 1981 Menasco's PhD thesis and paper \cite{Menasco1}. This was, for example, used for a rapid and simple algorithm that decides whether an alternating diagram represents the unknot in Section 7 of Lackenby's 2017 lecture notes \cite{LackNotes}. So the main open question was the complexity of link equivalence if links are given by connected prime alternating reduced diagrams. It was noticed independently by different authors that flypes enforce even more structure on such diagrams than the Tait flyping theorem states: this was captured in the proof of this theorem by Menasco and Thistlethwaite \cite{MT}, but also in the observations about flype circuits in the 1998 work of Sundberg and Thistlethwaite on the rate of growth of alternating links \cite{ST}, and about characteristic squares and product regions in 2004 Lackenby's work on volume of alternating links \cite{Lackenby}. It was suggested that this extra structure might be a key to the polynomial algorithm (for example, by Lackenby), and yet the algorithm was never written. The conjecture and some of the previosly known observations about non-prime, non-reduced diagrams were recorded in algorithmic context in 2023 paper by Samperton that concerned quantum invariants of knots \cite{Samperton}. 

In this paper, we give the main piece: the polynomial algorithm for links presented by connected prime alternating reduced diagrams. In addition to the extra structure imposed on a link diagram by flypes, we use known problems from graph theory and topological graph theory for steps of the algorithm. Note that link isotopy problem cannot be reduced to any variation of graph isomorphism problem: one link can have many diagrams that give rise to many graphs. Instead, we detect tangles that must be equivalent in certain sense due to the aforementioned extra structure, and use plane embeddings of graphs and ordered lists of colors to compare them inductively, in a certain hierarchy.

Note that our algorithm does not immediately apply to alternating links that are not given by alternating diagrams. If a non-alternating diagram is given, there is an algorithm that determines whether the link is alternating given by Howie \cite{Howie}. If it is, the algorithm produces its alternating diagram up to chirality. The complexity of Howie's algorithm is exponential, since it uses normal surface theory.

A restricted problem, equivalence for knots (\textit{i.e.} links with just one connected component) with alternating diagrams that satisfy some additional restrictions, was considered in \cite{KhVd}, and was shown to have low time and space complexity using different methods.

It is not yet known if 
\textsc{link equivalence} in general might be NP-hard. The following related problem was shown to be NP-hard by Koenig and Tsvietkova in \cite{KoTs}:

\vspace{0.08in}
\textsc{alternating sublink}: Given a diagram of a link $L$ and a positive integer $k$, does $L$ have a $k$-component sublink that is alternating?
\vspace{0.08in}

This can be compared with the fact that \textsc{subgraph isomorphism} is NP-hard \cite{Cook}.

In what follows, we start by recalling some basic facts from knot theory, graph theory and complexity theory in Section \ref{pre}. We then show in Section \ref{flypes} that there are at most $2n^2$ flypes between two prime, alternating, reduced and connected diagrams of a link $L$ with $n$ crossings (Proposition \ref{bound}). In Section \ref{algorithm}, we outline the topological facts necessary for the algorithm justification, and give the algorithm that compares two such diagrams to determine if they belong to the same link. We prove that the algorithm works exactly as it should in Theorem \ref{algworks}. In Section \ref{Poly}, we check that the number of steps in the algorithm is polynomial in $n$.

\section{Acknowledgements}

We thank Marc Lackenby and Arnaud De Mesmay for helpful discussions and correspondence. Among other things, Marc Lackenby pointed us to some of the related statements in his work \cite{Lackenby} back in 2010's, and Arnaud De Mesmay gave us some helpful background concerning algorithms for plane graphs. We also thank Sofia Dinets for the help with some figures. Tsvietkova was partially supported by NSF DMS-2005496 and NSF CAREER DMS-2142487.

\section{Preliminaries: alternating links, their diagrams, and graphs}\label{pre}


\subsection{Some basics from knot theory.} We think of link diagrams in $\mathbb{R}^2$ or $\mathbb{S}^2$ interchangeably, only with some language differences: e.g. one speaks of a side of a closed curve in $\mathbb{S}^2$ versus being inside or outside of it in $\mathbb{R}^2$. 

Let $L$ be an alternating link. Consider a connected prime alternating reduced diagram $D$ of $L$ with $n$ crossings. \textbf{In Sections \ref{pre}-\ref{algorithm}, we will only work with such diagrams.} Due to Menasco's results \cite{Menasco1}, the link $L$ must be topologically prime and non-split. Reduced alternating diagram must also have the minimal crossing number, say $n$, due to another Tait conjecture, proved by Kauffman, Murasugi, and Thistlethwaite \cite{Kauf, Murasugi, Thi}.

An \textit{edge} of a link diagram is its segment (an arc) from a crossing to the nearest crossing. Any link diagram with $n$ crossings has $2n$ edges.  A region of a link diagram (further \textit{diagram region}) is bounded by crossings and edges, and has none in its interior, \textit{i.e.} is complementary. A diagram region bounded by only two edges and two crossings is a \textit{bigon}. Any maximal sequence of bigons in $D$, where each two consecutive bigons share a crossing, is called a \textit{twist}. For a twist, a subset of consecutive bigons is a \textit{subtwist}. A subtwist can be maximal as well, \textit{i.e.} equal to a twist.

A 2-tangle in the terminology of Conway is a tangle with four free ends. 2-tangles in 3-space project to 2-tangles in $\mathbb{R}^2$ or $\mathbb{S}^2$. For 2-tangles in 3-space, we use ambient isotopy rel boundary, not allowing tangle ends to move. At the same time, there is usually no distinction between the four tangle ends unless we label the ends (we often will). Planar isotopy of 2-tangles implies their ambient isotopy rel boundary.

\subsection{Squares, tangles, and flypes.} Let us recall the terminology from Section 4 of Lackenby's paper \cite{Lackenby}. 

A \textit{square} is a simple closed curve that intersects a link diagram four times. If the intersections lie on four different edges of the link diagram, the square is \textit{essential}. A square is \textit{characteristic} if it is essential, does not separate off a single crossing on either side, and any other square can be isotoped of it, where the isotopy is planar (this means pull one square continuously of another, eliminating all their intersections; in the process any intersection of each square with $D$ must stay in the same diagram edge where it was, and no new intersections with $D$ can be introduced). Figure \ref{Nonchar} gives an example of two squares that are not characteristic. A \textit{characteristic collection} of squares is defined as follows: take one isotopy class of each characteristic square, and isotope them so that they are all disjoint.

\begin{figure}
\begin{subfigure}[b]{0.49 \textwidth}
\centering
\includegraphics[scale=1.1]{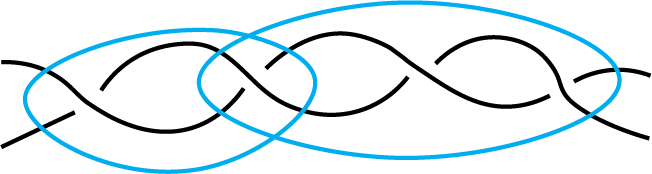}
\subcaption{}
\label{Nonchar}
\end{subfigure}
\begin{subfigure}[b]{0.49 \textwidth}
\centering
\includegraphics[scale=0.7]{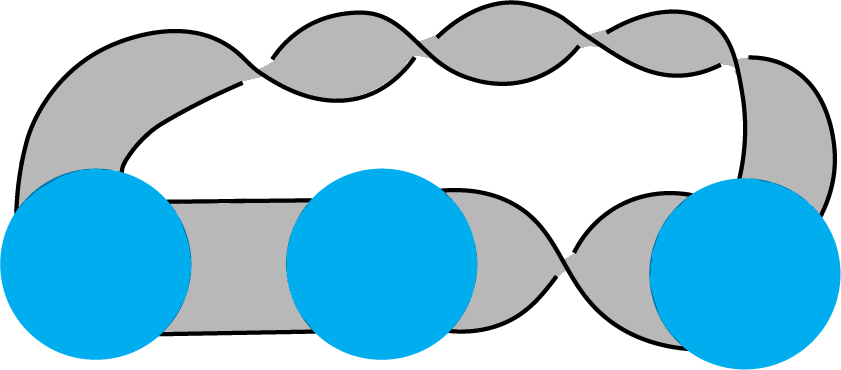}
\subcaption{}
\label{Product}
\end{subfigure}
\caption{(A): a twist from a link diagram (in black) and two squares that are not characteristic, since they cannot be isotoped of each other. (B): essential squares (in blue) and their complementary regions. The product region is colored in grey.}

\end{figure}

An essential square corresponds to a 2-tangle that has at least one crossing. Such a tangle can be the tangle $T$ in the definition of a flype. If $T$ has more than one crossing, we will say that the flype is \textit{non-trivial}. A trivial flype (for a tangle $T$ with 0 or 1 crossings) does not change an alternating link diagram. Hence in what follows, we will only need to consider non-trivial flypes.

The following definition is crucial for the algorithm. A version of it can be found in \cite{Lackenby}; we rephrase it slightly. 

Color a collection $\mathcal{C'}$ of essential squares in blue. Color the tangle each square in $\mathcal{C'}$ encircles on one of its sides in blue as well. Consider $D$ without blue tangles, and denote it by $F$. The regions bounded by arcs of $F$ (edges) and blue arcs on the boundary of squares can be colored in black and white in checkerboard fashion, since diagram regions of any reduced alternating diagram can be colored this way. If for color X (X is black or white), all regions of color X intersect crossings of $F$ and blue squares in only two connected components (e.g. in two crossings, or in a crossing and a blue arc), and if all regions of color X are pairwise connected through crossings or pairwise adjacent to the same blue square, call the union of all regions of color X a complementary \textit{product region} for $\mathcal{C'}$. Lackenby's observation in Lemma 8 of \cite{Lackenby} implies that a product region consists of twists and tangles beaded on two strands of a link as in Figure \ref{Product}. This construction was also considered earlier by Sundberg and Thistlethwaite \cite{ST}, and called a flyping circuit there. 

The above definition is for diagrams in $\mathbb{S}^2$: in particular, a product region can lie on any side of a square. Figure \ref{Product} shows one black product region for the blue squares.

Two distinct product regions can be \textit{adjacent} to each other as follows: a square of $\mathcal{C}$ might have two product regions on different sides of it (inside and outside, if $D$ is in $\mathbb{R}^2$), with the two regions being connected through some edges connected in $D$. For example, the link diagram in Figure \ref{link} (A) has adjacent product regions.

We will use the following lemma.

\begin{lemma}{\cite[Lemma 9 and Lemma 10]{Lackenby}}\label{lemma} Let $\mathcal{C}$ be a collection of disjoint non-parallel essential squares. It is a characteristic collection if and only if all of the following hold: 
\begin{enumerate}
\item any essential square in the complement of $\mathcal{C}$ lies in a product complementary region, or is parallel to a square of $\mathcal{C}$, or encloses a single crossing;

\item if two product complementary regions of $\mathcal{C}$ are adjacent, they have incompatible product structures;

\item no square of $\mathcal{C}$ encloses a single crossing.

\end{enumerate}
\end{lemma}

\subsection{Graphs and plane graphs.}\label{graphs} We recall some classical problems about graphs. 

Suppose we have a graph $\Gamma=(V, E)$ and function $f: V \rightarrow A$, where $A$ is a section of $\mathbb{Z_{+}}$, \textit{i.e.} the initial segment of $\{1, 2, 3, ...\}$. We then call $f$ a \textit{coloring function}, $\Gamma$ a \textit{vertex colored graph}, and $f^{-1}(i)$ the \textit{color class} of $i$ or of a vertex colored with $i$. We say that \textit{color valence} of $\Gamma$ is less or equal than $d$ with respect to $f$, if for every $v$ and $i$, either the number of neighbors or the number of non-neighbors of $v$ in $f^{-1}(i)$ is less or equal than $d$.

The following is a well-studied decision problem.

\vspace{0.08in}

\textsc{vertex colored graph isomorphism with bounded color valence.} Given two vertex colored graphs with bounded color valence, are they isomorphic?

\vspace{0.08in}

The isomorphism problem for colored graphs with bounded color valence is known to be solvable in polynomial time. For example, Prop. 4.5 of \cite{BL} assures that canonical form of a vertex colored graph with color valence at most $d$ can be computed in $|V|^{P(d)}$ steps, where $P(d)$ is a polynomial that depends only on $d$. The graph canonization problem is at least as computationally hard as the graph isomorphism problem, so \textsc{vertex colored graph isomorphism with bounded color valence} has computational complexity at most $|V|^{P(d)}$ as well. This can be compared with \textsc{graph isomorphism}, which is not known to lie in P (and thus the problem for colored graphs is potentially easier).

A \textit{pseudograph} is a graph that allows multiple edges between two vertices as well as \textit{loops}, \textit{i.e.} edges that connect a vertex to itself. \textsc{pseudograph isomorphism with bounded valence} can be reduced to \textsc{graph isomorphism with bounded valence} by adding extra vertices subdividing some edges. The resulting graph is called a \textit{subdivision graph}. If we start with a pseudograph with $n$ vertices and $2n$ edges, no loops, some double edges, and each vertex having valence 4, the subdivision graph is obtained in polynomial number of steps in terms of $n$, has linearly many edges and vertices in terms of $n$, and valence (and hence color valence) bounded above by 4. In particular, \textsc{graph isomorphism with bounded valence} is at least as hard as \textsc{pseudograph isomorphism with bounded valence}, and the former problem is in P as was proved by Luks \cite{Luks}. Similarly, \textsc{vertex colored graph isomorphism with bounded color valence} is at least as hard as  \textsc{vertex colored pseudograph isomorphism with bounded color valence}, and so both are in P.

Recall that a graph can be embedded in the plane if it can be drawn in the plane without crossings between edges. Such graphs are called \textit{planar}. There can be many plane embeddings of one planar graph. In some literature, a graph together with a fixed embedding is called a \textit{plane graph}. Such graphs are considered up to isotopy, which is a stronger notion than graph isomorphism: in particular, such isotopy is the map between the graphs that can be extended to a map from the plane to itself, isotopic to the identity. We will call the problem of comparing such graphs \textsc{plane graph isotopy}: it has been studied extensively in topological graph theory, and is known to be solvable in polynomial time. (For more on this problem, see, for example, \cite{Cori1, Cori2, CM}) .

Using the known algorithms, one obtains that comparing two plane vertex colored plane graphs or pseudographs up to isotopy preserving colors is in class P. For example, one way to do it for vertex colored connected plane graphs  $\Gamma_1$ and $\Gamma_2$ is as follows.

Pick an arbitrary vertex $v_1$ in $\Gamma_1$ and an edge $e_1$ adjacent to it. For each vertex $v_2$ in $\Gamma_2$, with the color $f(v_1)=f(v_2)$, and an incident edge $e_2$, consider a map from the pair $(v_1, e_1)$ to the pair $(v_2, e_2)$. For two planar embeddings to be isotopic, the cyclic order of the edges around equivalent vertices should be the same. Hence, by mapping a vertex and an edge, we uniquely (up to the change of clockwise/anticlockwise direction) determine how the edges adjacent to $v_1$ are mapped into the edges adjacent to $v_2$. Once the map has been extended to all edges adjacent to $v_1, v_2$, this also determines the map for vertices on the other end of the edges that were mapped. We first compare colors for these newly mapped vertices, and, if they coincide, proceed. Suppose such a vertex $v_1'$ in $\Gamma_1$ is mapped to $v_2'$ in $\Gamma_2$. Since for $v_1', v_2'$ at least one incident edge has already been mapped, we can use it to map other edges respecting the cyclic order again. Proceed by extending the map locally this way to all vertices and edges. This process has at most linearly many steps in terms of $|V|$. If for one choice of $(v_2, e_2)$, we were able to extend the map to the whole graph, there is an isotopy of plane graphs preserving colors. Otherwise, after trying all the possibles choices for $(v_2, e_2)$, there is no such isotopy. The total number of steps is at most $O(|V|^2)$.

\section{Bound on the number of flypes}\label{flypes}

Bounds on purely diagrammatic moves that take one link diagram to another one have been a rich topic of study. Previous work concerns mostly Reidemeister moves (e.g. \cite{CoLack, Lackenby2}), but it is perhaps useful to have a simple bound on the number of flypes as well. More importantly for us, we establish a relation between potential flypes and squares of characteristic collection in the proof of the following proposition. We will use this relation later for the polynomial algorithm.

\begin{proposition}\label{bound} Suppose we are given two diagrams of a link $L$, with each diagram being connected, prime, alternating, and reduced. Then at most $2n^2$ flypes are needed to transform one diagram to another, where $n$ is the crossing number.
\end{proposition}

\begin{proof}
Denote the two diagrams given in the theorem by $D$ and $D'$. As noted above, by one of Tait flyping conjectures (proved in \cite{Kauf, Murasugi, Thi}), the crossing number of these diagrams is the same. Consider a characteristic collection $\mathcal{C}$ for $D$.

For every tangle $T$ that can be flyped, there is an essential square that encircles it.  From Lemma \ref{lemma} (1), any essential square that is not in $\mathcal{C}$ up to planar isotopy either encircles a single crossing or lies in a complementary product region. If we want to count all possible flypes, we are not interested in any squares that encircle a single crossing (call them \textit{trivial} squares), since they give rise to trivial flypes. A square that lies in a product region encircles a subset of squares of $\mathcal{C}$ and subtwists of the product region. Hence in $D$, we can see all possible tangles that we can flype (\textit{i.e.} all potential tangles $T$ in the definition of a flype): such a tangle is either encircled by a  square of $\mathcal{C}$ or by a non-trivial essential square in a product region. 
Moving subtwists along a product region either does not change the diagram, or is equivalent to moving some squares from the characteristic collection. Hence we may ignore non-trivial essential squares that do not encircle any squares of $\mathcal{C}$ in the product region. Therefore, we may assume that in the definition of a flype, the tangle $T$ consists of squares $c_1, ..., c_k$, with $k\geq 1$, of $\mathcal{C}$, that share a product region $P$, and possibly some subtwists from $P$.

If $k=1$ and there are no subtwists from $P$ in $T$, call a flype that involves the tangle $T$ \textit{simple}. Otherwise call it a \textit{concatenation} flype. A concatenation flype can be substituted by  a sequence of simple flypes. Hence, in what follows, we may consider only simple flypes, and count them for an upper bound. Suppose we have a sequence of diagrams $D=D_1, D_2, D_3, ..., D_m=D'$, where $D_{i+1}$ is obtained from $D_i$ by a simple flype. We need to bound $m$. We already established above that the tangles in $D_1$ to which we can apply simple non-trivial flypes are in bijective correspondence with the squares of $\mathcal{C}$.

Once a flype is applied to a diagram, some squares of $\mathcal{C}$ might move with the flype along their product regions, but otherwise the squares do not change, and a characteristic collection is still a characteristic collection after a flype, in the resulting diagram. The proof of this is straightforward and can be found in the first paragraph of Lemma 4 in \cite{Lackenby}. In particular, $|\mathcal{C}|=|\mathcal{C}'|$. Therefore, the tangles to which we can apply simple non-trivial flypes are in bijective correspondence with the squares of $\mathcal{C}$ not just in $D_1$, but also in any $D_i, i=1, ..., m$.

Now let's count possible flypes that take us from a diagram $D$ to a diagram $D'$ of $L$. Place $D$ in $\mathbb{R}^2$. Each square $c$ of $\mathcal{C}$ has more than one crossing, but $c$ might have some other squares of $\mathcal{C}$ inside it. If $c'$ is inside $c$, call $c'$ a \textit{subsquare} of $c$, and is these squares are not parallel, call $c'$ a \textit{proper subsquare} of $c$. We may assume that for the square $c$, there is at least one crossing $x$ such that $x$ is not contained in any of its proper subsquares that are in $\mathcal{C}$. Hence the initial diagram $D$, we have up to $n$ squares of $\mathcal{C}$. There are also up to $n$ product regions in $D$, since each such region should contain at least one essential square or crossing.

 A simple flype can move one of $n$ squares of $\mathcal{C}$, say $c$, along its product region. Note that $c$ might have one product region, say $P_1$, on one side of $c$ (say, inside the square $c$), and another product region, say $P_2$, on the other side of $c$ (say, outside of $c$). Each product region in a link diagram has up to $n$ crossings outside of its squares, and hence $c$ can be moved up to $n$ times by a simple flype. Hence, when we choose square $c$ of $\mathcal{C}$, we have $n$ options; when we choose inside/outside product region for $c$, we have 2 options ($P_1, P_2$); and we can move $c$ at most $n$ times in each product region. This yields $2n^2$ possible simple flypes, each potentially taking us to one of the diagrams in the sequence $D=D_1, D_2, D_3, ..., D_m=D'$.

If there are more flypes, they must reverse what has been done by previous flypes, since we ran out of new options. Hence $m \leq 2n^2$.
\end{proof}

As corollary of the theorem, the sequence of flypes can serve as a polynomial certificate, putting alternating link equivalence in complexity class NP. Note that to specify a flype, it is enough to specify four edges for the square that encircles the tangle $T$, and specify a crossing outside of $T$. So to specify  a sequence $P(n)=2n^2$ flypes, a sequence of $4P(n)$ edges and $P(n)$ crossings is needed. The NP upper bound on complexity will be improved to P in the next sections.

\section{Algorithm for \textsc{alternating link equivalence}}\label{algorithm}

\subsection{Basis for the algorithm} The purpose of this subsection is to gather observations about alternating diagrams that yield a proof that our algorithm works. Some of these observations might be not new: implicitly, some of them seem to appear in \cite{MT, ST, Lackenby}. However these previous papers address questions different from ours, and the exact statement of Lemma \ref{lemma2} is not there.

The following facts either follow from the definition of a flype or are explained in Proposition \ref{bound} proof. 
\begin{enumerate}

\item Characteristic collection $\mathcal{C}$ of $D$ is invariant under flypes: its squares may move, but no squares are removed or added to $\mathcal{C}$ when a flype is performed.

\item A 2-tangle $T$ that lies inside a square $c$ of $\mathcal{C}$ does not change beyond planar isotopy of the tangle when a flype moving $c$ in $D$ is performed. (Note that $T$ is rotated with respect to $D$ after the flype.)

\item Every non-trivial flype can be substituted by a finite sequence of non-trivial simple flypes. The tangles to which simple flypes can be applied are in bijective correspondence with the squares of $\mathcal{C}$ (the squares encircle the respective tangles). A simple flype moves a tangle and respective square one crossing forward or backward in its product region $P$.

\item In particular, flypes cannot exchange squares of $\mathcal{C}$ in a product region.

\end{enumerate}

We call a side of a square of $\mathcal{C}$ \textit{inner} if it has no proper subsquares from $\mathcal{C}$ on that side. If there is only one square in a characteristic collection, it has two inner sides. 

For a product region $P$, call crossings it contains that do not lie in squares of $\mathcal{C}$, \textit{free crossings of $P$}.

From now on, we will place link diagrams $D, D'$ in $\mathbb{R}^2$ for convenience. 

Consider an extra curve that encircles all of $D$. Even though such a curve does not intersect $D$, we call it an \textit{exceptional square}.

\begin{rem}\label{orient} Suppose we have two diagrams $D, D'$, with one essential square each. Suppose two tangles inside these squares are planar isotopic, and two tangles outside the squares are planar isotopic as well. This does not mean that $D$ and $D'$ are planar isotopic: for example, $D$ and $D'$ and the links they represent can be related by a mutation, where one of the tangles in $D$ is rotated to obtain $D'$. Therefore, it is important how we connect the tangles in $D$ and $D'$. We will address this by labelling free ends of 2-tangles as follows: choose one free end if a 2-tangle and label it N, then proceed with labels W, S, E for the tangle ends counterclockwise. Assume the tangle is encircled by a square $c$: we label the ends twice, inside and right outside of the square. If such a tangle $T$ is isotopic to yet unlabelled tangle $T'$, transfer the labels by the isotopy, and put them both inside and outside of the square encircling $T$. Now suppose $T_1, T_2$ are 2-tangles lying in the same product region $P$ of $D$, with $k$ free crossings of $P$ between them, and the same is true for $T_1', T_2', P', D', k'$. Suppose $T_1$ is planar isotopic to $T_1'$, and  $T_2$ to $T_2'$. If $k<k'$, we can perform $k'-k$ flypes on $T_2'$ in $P'$ to make the number of free crossings between $T_1'$ and $T_2'$ equal to $k$. The flypes will rotate $T_2'$ for $180(k'-k)$ degrees. This might affect the labels at free ends of $T_2'$, but it is simple to compare the labels of $T_2, T_2'$ taking this into account.  
\end{rem}

\begin{lemma}\label{lemma2} Suppose $C_{k+1}$ is either an exceptional square, or lies in a characteristic collection $\mathcal{C}$ for a diagram $D$. Consider inner squares $C_1, ..., C_k$ inside $C_{k+1}$. Take the fragment of $D$ that is all of $D$ inside $C_{k+1}$ and outside of $C_1, ..., C_k$. Denote this fragment by $F$. Then the following holds:
\begin{enumerate}

\item Either all crossings of $F$ belong to one product region $P$ that is a common product region for $C_1, ..., C_k$, or there is no product region in $F$.

\item Any simple non-trivial flype inside $C_{k+1}$ either moves one of $C_i, i=1, ..., k$, or happens fully within one of $C_i$.

\item Define $C_1', ..., C_{k+1}', F'$ for $D'$ in the same way we defined $C_1, ..., C_{k+1}, F$ for $D$. Suppose all crossings of $F, F'$ belong to product regions $P, P'$ respectively. Suppose further that the number of free crossings of $P, P'$ is the same, the tangles in the sequences $C_1, ..., C_k$ and $C_1', ..., C_k'$ are pairwise planar isotopic, the cyclic order of isotopic tangles coincides in the two sequences, and the positions of N, W, S, E labels in $F, F'$ outside of the squares encircling $C_1, ..., C_k$ and $C_1', ..., C_k'$ are the same after the respective rotation, as in Remark \ref{orient}. Then the fragment of $D$ encircled by $C_{k+1}$ is ambient isotopic (rel boundary, if it is a tangle) to the fragment of $D'$ encircled by $C_{k+1}'$.

\end{enumerate}
\end{lemma}

Note that number $k$ for $P$ is called  \textit{weight} in \cite{ST}.

\begin{proof}
Item (1) follows from the definition of product region, definition of $F$, and Lemma \ref{lemma}(1). In particular, any product region is a sequence of twists. If there is a product region $P$ inside $C_{k+1}$, and a crossing inside $C_{k+1}$ that does not belong to $P$, then some part of $P$ is encircled by an essential square $c$. By Lemma \ref{lemma}(1), $c$ is then parallel to the square of $\mathcal{C}$ that fully lies in $C_{k+1}$. Then $P$ is not in $F$ by the definition of $F$.  

Item (2) follows from Lemma \ref{lemma}(2), item (3) in the discussion above the lemma, and (1) of this lemma. In particular, flypes only happen in product regions by item (3) in the discussion above the lemma. Two adjacent product regions are never compatible by Lemma \ref{lemma}(2). If there are two distinct product regions inside $C_{k+1}$, then one of them must lie inside $C_i, i=1,..., k$ due to (1) of this lemma. 

 Item (3) follows from items (2)-(4) in the discussion above this lemma, the definition of the product region, and the Tait flyping theorem. Indeed, ambient isotopy for our diagrams is equivalent to planar isotopy together with flypes. The only way to change $P$ by flypes is to apply some combination of simple flypes. Every simple flype moves a tangle bounded by $C_i$ forward or backward for one crossing in $P$ by (3) in the discussion above this lemma. In the process, the number of free crossings of $P$ stays the same, as do the tangles bounded by $C_1, ..., C_k$, and their cyclic order due to items (2) and (4) in the discussion above.
\end{proof}

\begin{rem}\label{comparegraphs} For two link diagrams $D, D'$, put a vertex at every crossing, and make diagram edges into edges between vertices. We obtain two planar pseudographs  $\Gamma, \Gamma'$, embedded in plane. Planar isotopy of $D, D'$ is equivalent to plane isotopy of $\Gamma, \Gamma'$ (such an isotopy is discussed in Subsection \ref{graphs}) together with the equivalence of crossing information for $D, D'$. Note that for alternating link diagrams, it is enough to compare over/under information for just one crossing of $D$ with such information for the respective crossing of $D'$: this determines similar information for the other crossings. If instead of all of $D, D'$ we have two tangles $T, T'$ encircled by squares $c, c'$, we can form the graphs similarly, but need to put an extra vertex at every free end of each tangle. We put such vertices on the intersection of $T$ and $c$ for convenience. See Figure \ref{tangle}. Planar isotopy of $T, T'$ is still equivalent to plane isotopy of the respective plane pseudographs together with the equivalence of crossing information for $T, T'$. 
\end{rem}

\begin{figure}
    \centering
    \begin{subfigure}[b]{0.49\textwidth}
              \centering
        \includegraphics[scale=0.9]{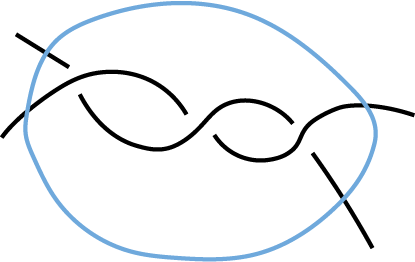}
        \subcaption{}
        \label{Tangle1}
    \end{subfigure}
       \begin{subfigure}[b]{0.49\textwidth}
        \centering
        \includegraphics[scale=0.9]{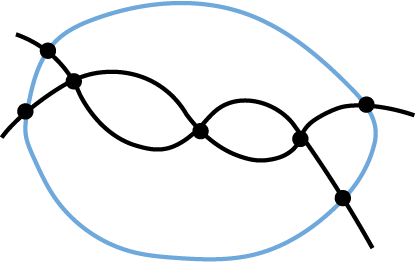}
       \subcaption{}
        \label{Tangle2}
    \end{subfigure}
    \caption{On the left, a 2-tangle (in black) encircled by a square (in blue). On the right, the respective pseudograph (in black).}
    \label{tangle}
\end{figure}

\subsection{Algorithm and its justification.}  In this subsection, we outline the algorithm, and prove that it works exactly as expected in Theorem \ref{algworks}.

\begin{alg}\label{alg}

\

\textit{Input:} two connected prime alternating reduced diagrams $D_1, D_2$ with crossing numbers $n_1, n_2$.

\textit{Output:} Yes or No, answering the question whether $D_1, D_2$ belong to the same link.

\vspace{-0.1cm}
\noindent\makebox[\linewidth]{\rule{6.5in}{0.4pt}}

\textit{Step 1.} If $n_1\neq n_2$, stop with output No.

\vspace{-0.1cm}
\noindent\makebox[\linewidth]{\rule{6.5in}{0.4pt}}

\textit{Step 2.} Detect a characteristic collection of squares $\mathcal{C}_i$ in $D_i, i=1, 2$. 

\vspace{-0.1cm}
\noindent\makebox[\linewidth]{\rule{6.5in}{0.4pt}}

\textit{Step 3.} Determine all inner sides for squares of $\mathcal{C}_i, i=1, 2$, and label them with $\theta_1$. Then label the squares that contain the already labelled squares but not other squares by $\theta_2$ on the side where the $\theta_1$ squares are. Continue inductively until all sides of all squares are labelled this way by the labels from the finite list $\theta_1, \theta_2, \theta_3, ...$. (Note: we might have a square labelled by $\theta_1$ on both sides, and then it is the only square in $D_i$.) See Figure \ref{link}(A) for an example of such labelling.

\vspace{-0.1cm}
\noindent\makebox[\linewidth]{\rule{6.5in}{0.4pt}}

\textit{Step 4a.}  Let $j=\theta_1$, and consider all squares labelled with $j$ at least on one side. For $j=\theta_1$, these are inner sides. On the inner side, there is a 2-tangle. Take a pair of such tangles $T_1, T_2$ encircled by squares $c_1, c_2$ in diagrams $D_1, D_2$ respectively.

\textit{Comparing two 2-tangles on inner sides of squares.} Make a pseudograph $\Gamma_i$ out of $T_i$ as in Remark \ref{comparegraphs}. Label each free end of $T_1$ by N, W, S, E, starting with N on some end, and continuing counterclockwise. Then compare the resulting plane pseudographs $\Gamma_i, i=1,2,$ up to plane isotopy. If isotopic, compare crossing information for one pair of crossings which yield equivalent vertices in $\Gamma_1, \Gamma_2$, \textit{i.e.} check if underpasses in $T_1$ correspond to underpasses in $T_2$. If the crossing information coincides, color the interior of a topological disk on the inner sides of $c_1, c_2$ in the same color. Label the free ends of $T_2$ by N, W, S, E according to the isotopy with $T_1$.

Once we compare all squares on the inner sides as outlined in the above paragraph, we have interior of every square in $D_1, D_2$ on the inner side colored. Make interior of each such square, say $c$ (together with the tangle the square encircles, say $T$) into a colored vertex $v$, otherwise leaving the diagram $D_i, i=1, 2$, untouched. Preserve the N, W, S, E labels of $T$, moving them from the ends of $T$ to the pseudograph edges incident to $v$, placing them near $v$ at those pseudograph edges (each pseudograph edge might get at most two labels this way, near each of its vertices). Denote the resulting hybrid of a colored and labelled pseudograph and a link diagram by $D'_i, i=1, 2,$ respectively. 

If at this point number of vertices of some color in $D_1$ is not the same as in $D_2$, stop with output No.

\vspace{-0.1cm}
\noindent\makebox[\linewidth]{\rule{6.5in}{0.4pt}}

\textit{Step 4b.} Repeat the following until we exhaust all values from $\theta_2, \theta_3, ...$ for $j$.

 Change $j$ to the next label in the list $\theta_2, \theta_3, ...$. Consider a square $c_i, i=1,2,$ in $D_i, i=1,2,$ that has a side labelled by $j$, and refer to this side as "inside" of $c$. By now, there are subsquares inside $c_i$ with colored interior, and with labels N, W, S, E each. 

\textit{Comparing two 2-tangles on $j$-side of squares.} If there is a product region $P_i$ inside one of $c_i$ but not another, the tangles are not equivalent. If there are product regions $P_i$ in each of $c_i, i=1,2$, compare the number of free vertices of $P_i$, and the lists of colored vertices inside $c_i$'s, up to cyclic order. If these numbers and lists are equal, also compare  N, S, W, E labels for vertices of the same color inside $c_i$ as in Remark \ref{orient}, and the crossing information (over/underpasses) for one pair of free crossings from $P_1, P_2$. If all coincides, $c_1$ and $c_2$ are considered equivalent. Lastly, if none of the $c_i$ contains a product region, make colored plane pseudographs out of 2-tangles inside $c_i$ as in Remark \ref{comparegraphs}, keeping the N, W, S, E labels as above. Compare these colored pseudographs up to plane isotopy preserving colors and labels. 

Consider all squares and their sides with label $j$ in $D_1$ and $D_2$. Compare each such square in $D_1$ with each such square in $D_2$ as outlined in the above paragraph. Give the interior of equivalent squares, on the side labelled by $j$, the same color, and labels N, W, S, E for their free ends as above. This new coloring absorbs all old colors and labels from the previous steps (for previous $j$) inside the squares we consider. Make these newly colored squares into new colored vertices, keeping N, W, S, E labels for them as above. This changes $D'_i, i=1, 2$. 

If at this point the number of vertices of some color in $D_1$ is not the same as in $D_2$, stop with output No.

Step 4 is partially illustrated in Figure \ref{link}(B) for $j=i_2$ for one of the diagrams, say $D_1$. 

\vspace{-0.1cm}
\noindent\makebox[\linewidth]{\rule{6.5in}{0.4pt}}

\textit{Step 5.} At this step, we have no more squares. We add an exceptional square to each of $D_1, D_2$ and use the comparison procedure from Step 4b, giving Yes or No as output. 


\end{alg}

\begin{figure}
    \centering
    \begin{subfigure}[b]{0.49\textwidth}
              \centering
        \includegraphics[scale=0.7]{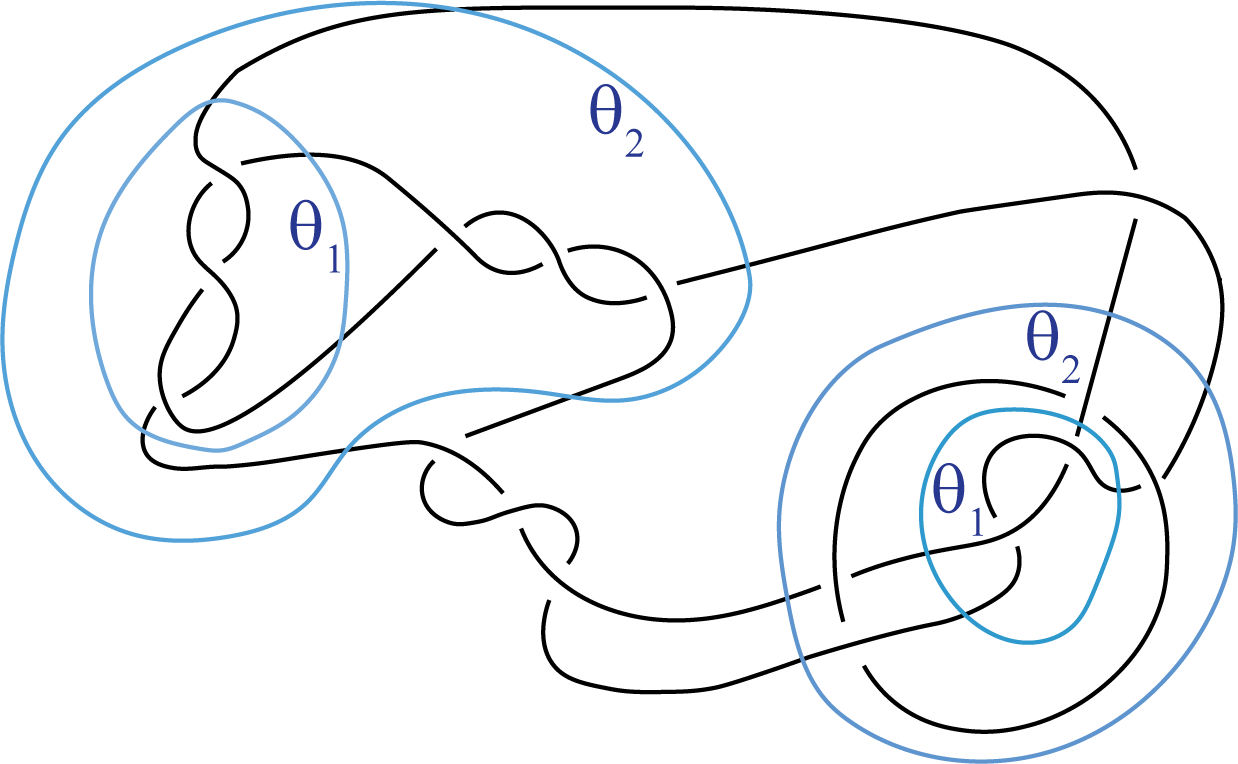}
        \subcaption{}
        \label{Link1}
    \end{subfigure}
       \begin{subfigure}[b]{0.49\textwidth}
        \centering
        \includegraphics[scale=0.7]{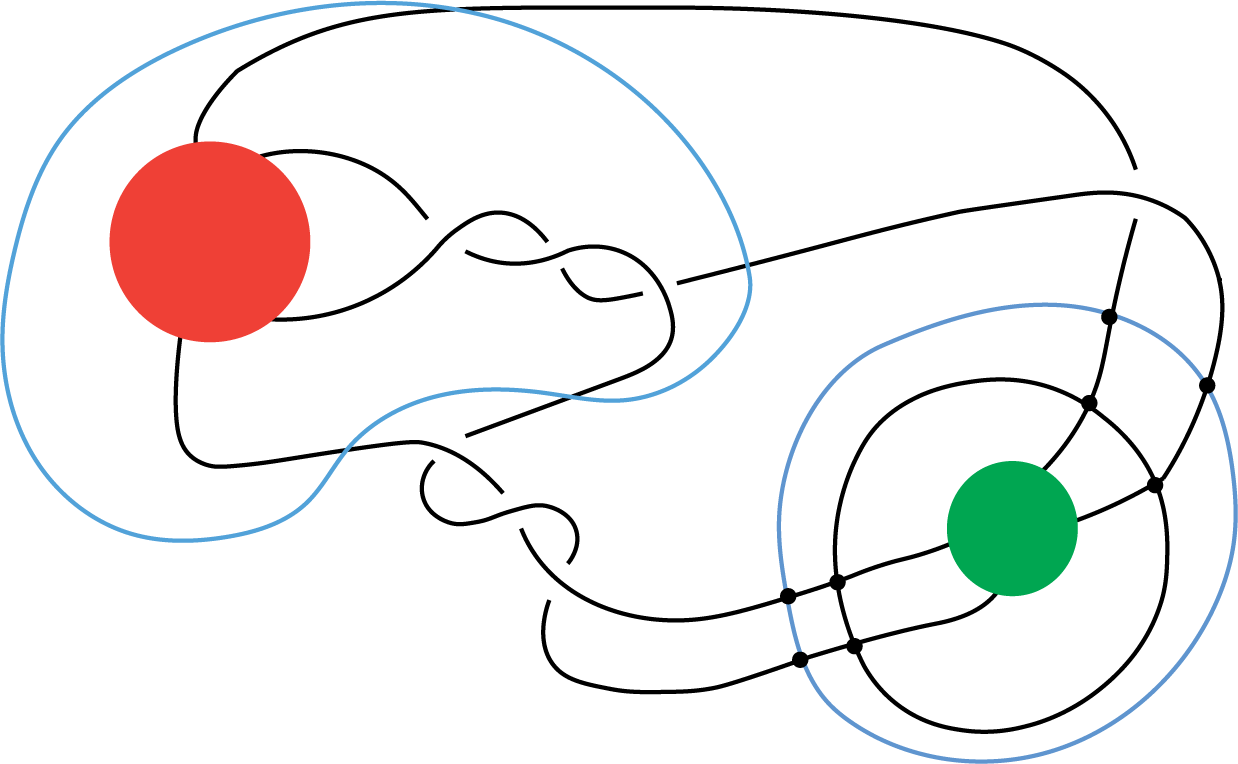}
       \subcaption{}
        \label{Link2}
    \end{subfigure}
    \caption{In Figure \ref{link}(a), there is an example of a link diagram \(D_1\) (in black) with a characteristic collection $\mathcal{C}_1$ of squares (in blue), where sides of squares are labeled by $\theta_1, \theta_2$. Substitute 2-tangles on the sides of squares of $\mathcal{C}_1$ labeled by $\theta_1$ with colored graph vertices (red and green), as in Step 4a. We obtain $D_1'$. Consider then the sides of squares labelled by $\theta_2$, considering them "interiors" of the squares for now. We see a product region inside the square on the left, and no product region inside the square on the right. We then turn the left square into a plane colored pseudograph. This changes $D_1'$: we obtain the one in Figure \ref{link}(b). All this is a part of Step 4b.}
    \label{link}
\end{figure}

\begin{thm}\label{algworks} Algorithm \ref{alg} gives $Yes$ if and only if $D_1, D_2$ represent the same link, up to ambient isotopy.
\end{thm}

\begin{proof} As noted above, by one of Tait flyping conjectures (proved in \cite{Kauf, Murasugi, Thi}), the crossing number of diagrams must be the same. This explains the stop and output No in Step 1. 

Step 2 and 3 proceed without any potential stops. 

By Tait flyping theorem, $D_1$ and $D_2$ are ambient isotopic if and only if they can be related by planar isotopy and a sequence of flypes.

\textbf{Claim 1} We can work on the diagram in steps for different $j$, \textit{i.e.} pairwise isotopy, planar or ambient rel boundary, of pairs of tangles encircled by squares of $\mathcal{C}_1, \mathcal{C}_2$ (where in each pair, one tangle is from $D_1$, and one from $D_2$, and such pairs exhaust all of $D_1, D_2$ and $\mathcal{C}_1, \mathcal{C}_2$), is equivalent to ambient isotopy of links represented by $D_1, D_2$.

\textit{Proof of Claim 1} Flypes do not move anything between interiors of different squares by Lemma \ref{lemma2}(2), and neither does planar isotopy. Hence, $D_1$ and $D_2$ are related by a sequence of flypes and planar isotopy if and only if the tangles $D_i$ encircled by squares on $j$-side are related by flypes and planar isotopy, as well as are, eventually, the complements of all square interiors in $D_i$.

To justify steps 4-5, we proceed inductively on $j=\theta_1, \theta_2, ...$., applying Lemma \ref{lemma2}(3) to squares of $\mathcal{C}_1, \mathcal{C}_2$ with sides labelled by the same $j$.  In particular, for two such squares $c_1, c_2$, let $C_{k+1}=c_1, C'_{k+1}=c_2$ in the lemma, and treat their sides labelled by $j$ as "inside". Then let the subsquares inside them, colored on previous step (for $j-1$) be $C_1, ..., C_k$ and $C'_1, ..., C'_k$ respectively to apply Lemma \ref{lemma2} (3) statement.  

By item (1) of the lemma, either $F, F'$ do not contain products regions, or all crossings of $F$ ($F'$) belong to one product region that is a common product region for $C_1, ..., C_k$ ($C_1, ..., C_k$). Recall that flypes can only happen in product regions. 

If there are no product regions in $F, F'$, then isotopy of plane pseudographs that respects crossing information is equivalent to planar isotopy of the respective 2-tangles by Remark \ref{comparegraphs}, and to ambient isotopy rel boundary of these tangles, since no flypes are possible. If one of $F, F'$ contains a product region, and one does not, the fragments of $D_i$ encircled by $c_i$ are not planar isotopic and the latter $c_i$ does not allow any flypes. Hence, by Tait flyping theorem, these cannot represent the same 2-tangles up to an ambient isotopy rel boundary.  If each of $F, F'$ contains a product region, we use Lemma \ref{lemma2} (1) to conclude that all of $F, F'$ is the product region, Lemma \ref{lemma2} (3) to compare the tangles encircled by $c_1, c_2$ and Remark \ref{orient} to compare N, W, S, E labels. Note that from working with previous $j$, we already know whether tangles in $C_1, ..., C_k$ and $C'_1, ..., C'_k$ are pairwise ambient isotopic rel boundary, and the isotopic ones were colored in the same color and have $N, W, S, E$ labels according to their isotopy. \end{proof}

\begin{rem}\label{gen} So far we considered connected prime alternating reduced diagrams, and hence only non-split prime alternating links. Once there is a polynomial algorithm to compare non-split prime alternating links with reduced diagrams, one can compare any two alternating links given by their alternating diagrams in polynomial time. This is based on the following facts, known since 1980's: (i) if an alternating diagram is not reduced, one can remove nugatory crossings by a 180-degree twist in ambient 3-space, and there are at most $n$ nugatory crossings; (ii) Menasco \cite{Menasco1} proved that connected alternating diagrams represent non-split links, therefore a separating 2-sphere in a link complement corresponds to a separating curve in the link projection; (iii) Menasco also proved that connected prime alternating diagrams represent prime links, therefore the essential 2-punctured sphere in the definition of a composite link corresponds to a curve that intersects link projection twice, and has knotted arcs on both sides. Hence, there is a rapid diagram-based algorithm that makes an alternating diagram reduced, and decomposes a split composite alternating link represented by a reduced alternating diagram into connected prime parts. While \cite{Samperton} concerns a very different topic from what we consider here, its subsection 2.1 gives an algorithmic version of the statements about reducing diagrams, and about passing from non-prime to prime diagrams. 
\end{rem}

\section{Algorithmic complexity}\label{Poly}

\begin{thm} The problem \textsc{alternating link equivalence} is in class P if the input is two alternating diagrams. In particular,  Algorithm \ref{alg} has a polynomial number of steps in terms of $\text{max}\{n_1, n_2\}$. 
\end{thm}
\begin{proof} In the view of Remark \ref{gen}, we only need to prove the second part of the theorem statement.

Complexities of steps 1-5 need to be added. If $n_1\neq n_2$, the algorithm stops at Step 1. So we may assume that $n_1=n_2=n$.

Step 1 is comparison of two natural numbers.

Step 2 can be subdivided in two substeps, each with polynomial complexity, and their complexities need to be added. 

First, detect all essential squares. For this, take all edges of $D_i, i=1, 2$. There are $2n$ of them for each diagram. Then take all (unordered) combinations of 4 edges out of them: there are  ${2n \choose 4}=O(n^4)$ of them. For each quartet of edges, check if the curve intersecting them is connected. If it is, it is a square. Check if there is more than one crossing inside and outside this square. If both hold, then the square essential. Therefore, all essential squares can be detected in polynomial number of steps in terms of $n$, and there are at most $O(n^4)$ of them.

Second, choose characteristic squares out of the essential squares. For this, for each pair of essential squares (there are at most $O(n^8)$ such pairs) check if they intersect as curves, and if the intersection contains a crossing. If this holds for a pair, the squares are not characteristic. If this does not hold for every pair, then it is a characteristic square. All together, this takes polynomial number of steps in terms of $n$.

Note that there are at most $n$ characteristic squares. Indeed, such a square $c$ must contain at least one crossing that is not contained in its proper subsquares of $\mathcal{C}_i$.

In Step 3, for each square of the characteristic collections $\mathcal{C}_i, i=1, 2$, check if it has a subsquare: e.g., by checking if a square shares all crossings with another square. For each of at most $4n^2$ pairs of sides of squares, this requires checking at most $n$ crossings, and hence yields at most $4n^3$ steps. This way we detect and label all $i_1$ sides of squares of $\mathcal{C}_i, i=1, 2$. Then for each $i_1$ square, consider its non-labelled side, and check whether it shares crossings with any other non-labeled side of another square. Again there are up to $4n^2$ pairs of sides of squares, and up to $n$ crossings for each pair to check. We have at most $4n^3+4n^3$ steps to detect $i_1$ and $i_2$ squares. Continue at most $n$ times, get at most $n$ by $4n^3$ steps to label all sides of squares by $i_1, i_2, i_3, .... i_k, k\leq n$.

Step 4 is repeated until every two tangles encircled by the squares of $\mathcal{C}$ are compared, where a tangle may lie on one of two sides of a square. So we have at most $2n$ times $2n$ pairs of tangles or product regions to compare. When we create a pseudograph from a 2-tangle, we do this in polynomial number of steps. In the resulting pseudographs, there are no loops and there are at most two double edges between two vertices. Due to properties of a link diagram, there are at most $n$ pseudograph vertices (or $2n$ for a subdivision graph), the valence of each vertex is at most $4$, there are up to $2n$ edges for pseudograph (or $3n$ for a subdivision graph). The total number of colors is at most $n$, since we have at most $n-1$ distinct tangles, and can add one more color that is used for any pseudograph vertex that has not been obtained from a colored tangle.

In Step 4a, for $j=i_1$, the tangle comparison is polynomially reduced to \textsc{plane pseudograph isotopy with bounded valence}, which has polynomial complexity as explained in subsection \ref{graphs}.

In Step 4b, for $j=i_2, ..., i_k$, the comparison for each pair of product regions is reduced in some cases to comparison of two ordered lists of colors, comparison of two numbers (\textit{i.e.} the weights of product regions), and comparing N, W, S, E labels for each of at most $n$ equivalent pairs of equivalent tangles. This can be done in polynomially many steps. In other cases, it is reduced to \textsc{vertex colored plane pseudograph isotopy with bounded color valence}. It is known to have polynomial complexity as explained in subsection \ref{graphs}. Additionally, some edge labels need to be compared, which is still polynomial (N, W, E, S). Denote the larger of these polynomials by $P$.

The polynomials for these two substeps must be multiplied by $4n^2$, \textit{i.e.} the upper bound on number of pairs of tangles or product regions to compare, and added.

Step 5 complexity is bounded by the same polynomial $P$ as Step 4b.
\end{proof}

\bibliographystyle{amsplain}
  \bibliography{./FlypesBibFile}{}

\vspace{0.1in}

Touseef Haider\\
Department of Mathematics and Computer Science\\
Rutgers University-Newark \\
101 Warren Street, Newark, NJ  07102, USA\\
th582@newark.rutgers.edu

\vspace{0.1in}

Anastasiia Tsvietkova\\
Department of Mathematics and Computer Science\\
Rutgers University-Newark \\
101 Warren Street, Newark, NJ  07102, USA\\
a.tsviet@rutgers.edu

\end{document}